\documentclass[preprint,12pt]{elsarticle}

\usepackage{amsmath,amssymb,amsthm,amsfonts,mathtools,bm}
\usepackage{graphicx}
\usepackage{tikz}
\usetikzlibrary{decorations.pathreplacing,angles,quotes}
\usepackage{esint}
\usepackage{physics}
\usepackage{comment}
\usepackage{bigints}
\usepackage{etoolbox}
\usepackage{color}

\usepackage{hyperref}
\usepackage{cleveref}

\usepackage{arydshln}
\usepackage{thmtools,thm-restate}

\hypersetup{
  colorlinks=true,
  linkcolor=blue,
  citecolor=blue,
  urlcolor=blue,
  linktocpage=true
}

\theoremstyle{plain}
\newtheorem{theorem}{Theorem}[section]
\newtheorem{lemma}[theorem]{Lemma}

\newtheorem{proposition}[theorem]{Proposition}

\numberwithin{equation}{section}



\DeclareMathOperator*{\dist}{dist}

\begin{document}

\begin{frontmatter}

\title{Doubling Argument of the Hessian Estimate for the Special Lagrangian Equation on General Phases with Constraints}

\author[inst1]{Cheuk Yan Fung\corref{cor1}}
\cortext[cor1]{Corresponding author}
\ead{pifung@umd.edu}

\affiliation[inst1]{
    organization={Department of Mathematics, University of Maryland},
    addressline={4176 Campus Dr}, 
    city={College Park},
    postcode={MD 20742},
    country={United States}
}

\begin{abstract}
In this paper, we establish a doubling argument to obtain Hessian estimates for the special Lagrangian equation under general phase with constraints. In particular, our approach does not rely on the Michael--Simon mean value inequality. As an intermediate step, we also establish Alexandrov-type theorems, which may be of independent interest.
\end{abstract}


\end{frontmatter}

\section{Introduction}
In this paper, we prove the Hessian estimate for the special Lagrangian equation on general phases with constraints, first introduced by Zhou~\cite{Zhou2021}. Our result extends Shankar's argument~\cite{Shankar2024HessianEF} which also employs the doubling method. 

Let $u$ be a smooth solution of the special Lagrangian equation in dimension $n\geq 3$ with a phase constant $\theta \geq 0$ given by
\begin{equation}
\label{SLE}
    \arctan \lambda_1 + \arctan \lambda_2 + \dots + \arctan \lambda_n = \theta,
\end{equation}
where $\lambda_i $ are the eigenvalues of $D^2u$. 

Hessian estimates for the special Lagrangian equation have been extensively studied and established in various contexts. In dimension three, Warren--Yuan~\cite{WarrenYuan2009} obtained the estimate for the critical phase, and later Warren--Yuan \cite{a359f720-81a7-3f4b-b3f0-6e19654fb57a} further extended the results to both the critical and supercritical phases. Wang--Yuan~\cite{Wang2014} subsequently developed the estimates for the critical and supercritical phases for general dimensions $n\geq 3$. Chen--Warren--Yuan~\cite{Chen2009} proved the estimate for convex solutions. Warren--Yuan~\cite{WarrenYuan2008} also derived the Hessian estimate under the assumptions
\[3+(1-\varepsilon)\lambda_i^2+2\lambda_i\lambda_j \geq 0,\text{ } 1\leq i,j\leq n,\] for some $\varepsilon >0$, and $|Du(x)|\leq \delta(n) |x|$. Ding~\cite{Ding2023} later established the Hessian estimate with the assumption 
\[(3+\varepsilon)+(1+\varepsilon)\lambda_i^2+2\lambda_i\lambda_j \geq 0, \text{ } 1\leq i,j\leq n,\] for some $\varepsilon >0$. Finally, Zhou~\cite{Zhou2021} obtained the Hessian estimate on general phases with constraints. Counterexamples below the critical phase were constructed by Mooney--Savin~\cite{MooneySavin2024}, Nadirashvili--Vl{\u{a}}du{\c{t}}~\cite{NadirashviliVladut2010}, and Wang--Yuan~\cite{WangYuan2013}. 

A natural question is whether we can find a pointwise argument for the a priori Hessian estimates of the special Lagrangian equation. In fact, Chen--Warren--Yuan~\cite{Chen2009}, Wang--Yuan~\cite{Wang2014}, and Zhou~\cite{Zhou2021} established integral proofs for the Hessian estimates of the special Lagrangian equation, relying on the Michael--Simon mean value inequality~\cite{MichaelSimon1973} and the Jacobi inequality. Recently, Shankar~\cite{Shankar2024HessianEF} developed a pointwise, doubling argument to obtain Hessian estimates for the special Lagrangian equation in various settings. In particular, the argument applies to solutions with the critical and supercritical phases, and to convex solutions \cite{Chen2009, Wang2014}. He made use of a Korevaar-type pointwise calculation that avoids employing the Michael--Simon mean value inequality. In this article, we adapt the pointwise, doubling method to the more general setting introduced by Zhou~\cite{Zhou2021}, thereby extending the approach to these constraints.

The study of Hessian estimates for the special Lagrangian equation is intimately related to that of the interior gradient estimate for the minimal surface system. Specifically, the Hessian estimate for the special Lagrangian equation can be viewed as a particular case of the gradient estimate for minimal surface systems, when the minimal surface is represented as a gradient graph. Bombieri–De Giorgi–Miranda \cite{Bombieri1969} established an a priori gradient estimate for graphical minimal hypersurfaces, which was later simplified by Trudinger \cite{Trudinger1972}. The key ingredients are the Jacobi inequality $\Delta a \geq |\nabla a|^2$ and the Michael-Simon mean value inequality. Korevaar \cite{Korevaar1987} subsequently established a pointwise argument for the gradient estimate without using the Michael-Simon mean-value inequality. A natural question is whether such an argument can be adapted for higher codimension minimal surfaces. An example is that Dimler \cite{Dimler2023} obtained the gradient estimate for classical solutions of the minimal surface system assuming the area decreasing condition and that all but one component have small $L^\infty$ norm. 

The doubling method provides an alternative approach to establishing Hessian estimates without using the Michael-Simon mean value inequality and can also be used to derive interior regularity results. The first step is to prove an Alexandrov-type theorem. Next, the uniform norm of the solution is flattened so that Savin's small perturbation theorem (\cite{Savin2007}, Theorem 1.3) can be applied to obtain Hessian estimates in a small ball. Finally, by employing the Jacobi inequality, one can derive the doubling inequality, which allows the Hessian estimate in a larger ball to be controlled by the estimate in the smaller ball.

Several works have constructed doubling inequalities to establish Hessian estimates and interior regularity results. Qiu \cite{Qiu2024} derived the Hessian estimate of the sigma-2 equation in dimension $3$. Shankar-Yuan \cite{10.4007/annals.2025.201.2.4} established the Hessian estimate for the sigma-2 equation when $n=4$. Shankar-Yuan \cite{ShankarYuan2024} proved the interior regularity of strictly convex solutions of the
Monge-Amp$\grave{e}$re equation. As discussed above, Shankar~\cite{Shankar2024HessianEF} extended this doubling framework to the special Lagrangian equation in several settings. In this article, we further adapt the doubling method to handle the general constraints introduced by Zhou~\cite{Zhou2021}.

In related work, Trudinger \cite{Trudinger1980} established a doubling inequality in the context of the Harnack inequality for weak solutions $u \in W^{2,n}$ of general second order elliptic quasilinear equations. Caffarelli-Wang \cite{CaffarelliWang1993} made use of a Harnack inequality approach to give an alternative proof of the $C^{1,\alpha}$ regularity of Lipschitz solutions to 
\[F(v,II) = 0,\]
where $II$ is the second fundamental form of a hypersurface $S$ in $\mathbb{R}^{n+1}$ and $v$ is its normal. 

We say that $\lambda(D^2u) \subset A$ if the eigenvalues $(\lambda_1,\dots, \lambda_n)$ of $D^2u$ belong to $A \subset \mathbb{R}^n$. 
 Recall that for $\lambda = (\lambda_1, \dots, \lambda_n)$, the $k$-th elementary symmetric polynomial is given by $\sigma_k(\lambda) = \sum_{1 \le i_1 < \dots < i_k \le n} \lambda_{i_1} \dots \lambda_{i_k}$ for $1 \le k \le n$, with $\sigma_0 \equiv 1$.

Let 
\begin{equation}
\label{cone}
    \Gamma_{n-1}^c = \{\Lambda \in \mathbb{R}^n| \text{ }\sigma_{n-1} > c\lambda_2 \cdot \lambda_3\cdot ... \cdot \lambda_n, \text{ } \lambda_{n-1} >0 \}, 
\end{equation}
where $(\lambda_1,\dots,\lambda_n)$ is an arrangement of $\Lambda$ such that $\lambda_1 \geq \lambda_2 \geq \dots \geq \lambda_n$ and $c$ is a constant. We denote the closure of the cone $\Gamma_{n-1}^c$ by $\overline{\Gamma_{n-1}^c}$. From now on, we will always assume $\lambda_1 \geq \lambda_2 \geq  \dots \geq \lambda_n$. 

The following theorems were proved by Zhou \cite{Zhou2021} using the Michael-Simon mean value inequality. We establish a compactness argument without using it. 
\begin{theorem}
\label{thm1.1}
    Let $n\geq 3$ and let $u$ be a smooth solution to the special Lagrangian equation (\ref{SLE}) in $B_2(0)$ with the phase constant $\theta \geq 0$. Suppose that $\lambda(D^2 u) \subset  \overline{\Gamma^{(n-2)/2}_{n-1}}$, then  
    \begin{equation}
        |D^2u(0)| \leq C(n, \|u\|_{C^{0,1}(B_1(0))},\theta).
    \end{equation}
\end{theorem}
In dimension three, the result can be refined to the following:
\begin{theorem}
\label{thm1.2}
    Let $n =3$ and let $u$ be a smooth solution to the special Lagrangian equation (\ref{SLE}) in $B_2(0)$ with the phase constant $\theta \geq 0$. Suppose that $\sigma_2(D^2u) \geq (3/5-\varepsilon) \lambda_2 \lambda_3$ for some constant $\varepsilon >0$, where $\lambda_1 \geq \lambda_2 \geq \lambda_3 $ are eigenvalues of $D^2 u$, then we have 
    \begin{equation}
        |D^2u(0)| \leq C( \|u\|_{C^{0,1}(B_1(0))},\theta,\varepsilon).
    \end{equation}
\end{theorem}
Note that $\lambda_2 $ can be negative in this case (only if $\varepsilon>3/5$).

This paper follows the outline of Shankar \cite{Shankar2024HessianEF}. In our settings, Jacobi inequalities have already been established by Zhou \cite{Zhou2021}. The new contribution of this paper is the proof of the Alexandrov-type theorem. Then we can apply the doubling inequality argument developed by Shankar \cite{Shankar2024HessianEF}, with only minor modifications, to obtain the Hessian estimates.

\textbf{Outline:} The structure of the paper is as follows. In Section \ref{section preliminary}, we show some preliminary results. In particular, we show the gradient estimates for our settings in Lemma \ref{grad_est} in which we modify the proof of the gradient estimate for large phase by Yuan \cite{private_note}. In Section \ref{section partial regularity}, we establish the Alexandrov-type theorem for our settings. In Section \ref{section doubling inequality}, we establish the doubling inequalities for our settings. In fact, the argument of Shankar \cite{Shankar2024HessianEF} works with minor adjustments. In Section \ref{section main theorem proof}, we complete the proofs of Theorem \ref{thm1.1} and Theorem \ref{thm1.2}. In Section \ref{appendix}, we restate the proof of the gradient estimate of Yuan \cite{private_note}.

\section{Preliminary}
\label{section preliminary}
The following are the standard notations and results of the special Lagrangian equation. Let $u$ be a smooth solution to the special Lagrangian equation (\ref{SLE}). The metric
     \begin{equation*}
     g = (g_{ij}) = I + D^2uD^2u
     \end{equation*}
     is the induced metric on the surface $(x, Du (x)) \in \mathbb{R}^n\times \mathbb{R}^n$. 
It is well known that 
\begin{equation*}
    \Delta_g = \sum_{i,j}g^{ij}\partial_i\partial_j,
\end{equation*}
and
\begin{equation*}
    \sum_{i,j}g^{ij}u_{ijk} = 0.
\end{equation*}
Harvey-Lawson \cite{Harvey1982} showed that the graph $(x,Du(x))$ is a volume minimizing submanifold of $\mathbb{R}^n \times \mathbb{R}^n$. 

The following three lemmas describe some properties of the eigenvalues of $D^2u$, where $u$ satisfies the assumption of Theorem \ref{thm1.1} or Theorem \ref{thm1.2}.
\begin{lemma}[\cite{Zhou2021} Lemma 4.1] \label{n>=3eigenvalues}
Let $n\geq 3$ and let $\lambda_1 \geq \lambda_2 \geq \dots  \geq \lambda_n $ be $n$ real numbers such that $(\lambda_1,\dots,\lambda_n) \in \Gamma_{n-1}^{(n-2)/2}$ and $\sum_{i=1}^n \arctan \lambda_i \geq 0$. Then 
\[
\lambda_i > -\lambda_n \quad \text{for all } 1\leq i \leq n-1.
\]
A similar property holds with $\overline{\Gamma_{n-1}^{(n-2)/2}}$ instead of $\Gamma_{n-1}^{(n-2)/2}$ (with $\geq$ instead of $>$).
\end{lemma}
For the following lemma, part (a) is from \cite{Zhou2021} Lemma 3.1, and parts (b), (c) and (d) are taken from the proof of \cite{Zhou2021} Proposition 3.2.    
\begin{lemma}[\cite{Zhou2021} Lemma 3.1 and Proposition 3.2] 
\label{n=3eigenvalues}
Let $\lambda_1 \geq \lambda_2 \geq \lambda_3$ be three real numbers such that $\sigma_2(\lambda_1,\lambda_2,\lambda_3) \geq (3/5-\varepsilon)\lambda_2 \lambda_3$ for some $\varepsilon>0$ and $\arctan \lambda_1 + \arctan \lambda_2 +\arctan \lambda_3 \geq 0$. Then the following holds:
\begin{enumerate}[(a)]  
    \item If $\lambda_2 \geq 0$, then $\lambda_1 \geq \lambda_2 \geq -\lambda_3$.
    \item If $\lambda_2 < 0$, then $0 < \lambda_1 < \max(1,\varepsilon)$.
    \item If $\lambda_2 < 0$, then $|\lambda_2| \leq \lambda_1$ and $|\lambda_3| \leq \lambda_1$.
    \item If $\lambda_2 < 0$, then $\sigma_1(\lambda_1,\lambda_2,\lambda_3) \geq 0$.
\end{enumerate}
\end{lemma}

\begin{lemma}
\label{n=3eigenvalues2}
  Let $\lambda_1 \geq \lambda_2 \geq \lambda_3$ be three real numbers such that $\sigma_2(\lambda_1,\lambda_2,\lambda_3) \geq (3/5-\varepsilon)\lambda_2 \lambda_3$ for some $\varepsilon>0$ and $\arctan \lambda_1 + \arctan \lambda_2 + \arctan \lambda_3 \geq 0$. Suppose that $\lambda_2<0$, then
    \begin{equation}
        |\lambda_2| \geq \frac{|\lambda_3|}{2/5+\varepsilon} = C|\lambda_3|.
    \end{equation}   
\end{lemma}
\begin{proof}
      Grouping the terms in $\sigma_2(D^2u) \geq (3/5-\varepsilon)\lambda_2 \lambda_3$,
    we have 
    \begin{equation*}
        ((2/5+\varepsilon)\lambda_2+\lambda_1)\lambda_3 \geq -\lambda_1 \lambda_2.
    \end{equation*}
    If $(2/5+\varepsilon)\lambda_2+\lambda_1 > 0,$ then 
    $
        \lambda_3 \geq-\frac{\lambda_1\lambda_2}{(2/5+\varepsilon)\lambda_2+\lambda_1}>0,
    $ which contradicts $\lambda_2 \geq \lambda_3$. If $(2/5+\varepsilon)\lambda_2+\lambda_1 = 0$, then $
        0 \geq - \lambda_1 \lambda_2,
    $ which contradicts $\lambda_1 >0 $ and $\lambda_2<0$.
    Hence, the only possible case is $
        (2/5+\varepsilon)\lambda_2+\lambda_1<0, 
    $ which yields  $
        \lambda_2 < \frac{-\lambda_1}{2/5+\varepsilon}.
    $
    Subsequently, by Lemma \ref{n=3eigenvalues}(c), we have 
    \begin{equation*}
        |\lambda_2| > \frac{|\lambda_1|}{2/5+\varepsilon} \geq\frac{|\lambda_3|}{2/5+\varepsilon}.
    \end{equation*}
\end{proof}
The gradient estimate for the special Lagrangian equation with large phase, i.e., $|\theta| \geq (n-2)\pi/2$, was established by Warren and Yuan \cite{a359f720-81a7-3f4b-b3f0-6e19654fb57a} (Theorem 1.3), and later refined by Yuan in his private note \cite{private_note} (Theorem 5.1.1). The gradient estimates for the settings Theorem \ref{thm1.1} and Theorem \ref{thm1.2} follow from a slight modification of Yuan's proof. The key observation is that if we have the conditions that $\lambda_n$ is the only eigenvalue that can be negative and $\lambda_i\geq |\lambda_n| $, for $1\leq i \leq n-1$, then we can obtain the same result using Yuan's method \cite{private_note}.  For the reader’s convenience, Yuan's proof is restated in Theorem \ref{grad_yuan}.

\begin{lemma}
\label{grad_est}
Let $u$ be a smooth solution to the special Lagrangian equation (\ref{SLE}) with $\theta\geq 0$ on $B_{R}(0) \subset \mathbb{R}^n$. 
\begin{enumerate} [(a)]
    \item (Gradient estimate for Theorem \ref{thm1.1}.) Suppose that $n\geq 3$ and $\lambda(D^2 u) \subset  \overline{\Gamma^{(n-2)/2}_{n-1}}$, then we have 
\begin{equation}
    |Du(0)| \leq C(n)\frac{osc_{B_{R}(0)}u}{R}.
\end{equation}
\item  (Gradient estimate for Theorem \ref{thm1.2}.) Suppose that $n = 3$ and $\sigma_2(D^2u) \geq (3/5-\varepsilon) \lambda_2 \lambda_3$ for some constant $\varepsilon >0$ where $\lambda_1 \geq \lambda_2 \geq \lambda_3 $ are eigenvalues of $D^2 u$, then we have 
\begin{equation}
    |Du(0)| \leq C(\varepsilon)\frac{osc_{B_{R}(0)}u}{R}.
\end{equation}
\end{enumerate}

\end{lemma}
\begin{proof}
(a) By Lemma \ref{n>=3eigenvalues}, $\lambda_n$ is the only eigenvalue that can be negative and $\lambda_i\geq |\lambda_n| $, for $1\leq i \leq n-1$. Note that these two conditions are the only properties of $\theta$ that Yuan used in the proof of the gradient estimate for large phase. See Theorem \ref{grad_yuan} for Yuan's proof. 

(b) The proof follows from a slight modification of Yuan's proof (Theorem \ref{grad_yuan}). As in the original proof, we consider the test function
\begin{equation*}
    w = (1-|x|^2)|Du| + \frac{n}{M}u^2.
\end{equation*}
Suppose that the maximum point $x^*$ lies on the boundary, we may use Yuan's original proof (Theorem \ref{grad_yuan}). 

Suppose that the maximum point $x^*$ lies in the interior, we may assume that there exists a $k\in \{1,\dots,n\}$ such that $u_k \geq \frac{|Du|}{\sqrt{n}}$ at $x^*$.
As in Theorem \ref{grad_yuan}, we know that $u_{kk} <0$.  If $\lambda_2(x^*) \geq 0$, then $\lambda_3$ is the only eigenvalue that can be negative. Also, by Lemma \ref{n=3eigenvalues}(a), $\lambda_i+\lambda_3 \geq 0$ for $i= 1,2$. Again, these two conditions are the only properties of $\theta$ that Yuan used in the proof of the gradient estimate for large phase. Hence, we can derive the gradient estimate. We may now assume $\lambda_2(x^*) < 0$. In this case, either $k= 2$ or $k=3$.
By Lemma \ref{n=3eigenvalues2}, we see that $|\lambda_2| \geq C|\lambda_3|$ and by definition $|\lambda_2| \leq |\lambda_3|$.
Also, notice that by Lemma \ref{n=3eigenvalues}(b) and (c), $|\lambda_i|\leq \max(1,\varepsilon)= C(\varepsilon)$.
Then using diagonal coordinates at $x^*$, we have
\begin{equation}
\label{2_1_2}
    C(\varepsilon) \leq g^{ii}= 1/(1+\lambda_i^2)\leq 1.
\end{equation}
At the maximum point $x^*$, a direct computation shows
\begin{equation}
\label{maxprin_testfunction}
    0 \geq \Delta_g w = \Delta_g \eta |Du| + 2g^{\alpha\beta} \eta_\alpha |Du|_\beta+ \eta \Delta_g |Du| + 2\frac{n}{M}g^{\alpha \beta}u_\alpha u_\beta+ 2\frac{n}{M}u\Delta_g u.
\end{equation}
We look at each term separately. We may handle the third and fourth terms exactly the same as in Yuan's proof. We will use diagonal coordinates at $x^*$. Using (\ref{2_1_2}) for the first term, we have
\begin{equation*}
    \Delta_g \eta = -2\sum_{\alpha} g^{\alpha\alpha} \geq -C.
\end{equation*}
Hence,
\begin{equation}
\label{2_6_2}
    \Delta_g \eta |Du| \geq -C|Du|.
\end{equation}
Using (\ref{2_4}), (\ref{2_1_2}) and the bound $M\leq u \leq 2M$ for the second term, we have
\begin{equation}
\label{secondterm}
\begin{split}
\sum_{\alpha,\beta}2g^{\alpha\beta} \eta_\alpha |Du|_\beta &= \sum_{\alpha}2g^{\alpha\alpha}(-2x_\alpha)(\frac{2x_\alpha|Du|+2\frac{n}{M}uu_\alpha}{\eta}) \\
&\geq -C(\sum_\alpha x_\alpha^2)\frac{|Du|}{\eta}-C\frac{|Du|}{\eta}\\
&\geq -C\frac{|Du|}{\eta}.
\end{split}
\end{equation}

Using Lemma \ref{n=3eigenvalues2}, the inequality $u_{11}>0$ and the inequality $|\lambda_2| \leq |\lambda_3|$ for the fifth term, we have
\begin{equation*}
    \Delta_g u = \sum_\alpha g^{\alpha\alpha}u_{\alpha\alpha} = \sum_\alpha \frac{u_{\alpha\alpha}}{1+u_{\alpha\alpha}^2} \geq \lambda_2 + \lambda_3 
    \geq C\lambda_k = Cu_{kk}.
\end{equation*}
Using  (\ref{2_4}) and the fact that $u_k \geq |Du|/\sqrt{n}$, we see that 
\begin{equation}
\label{fifthterm}
    \Delta_g u\geq  Cu_{kk} = -C\frac{|Du|(\eta_k|Du|+2\frac{n}{M}u u_k)}{u_k\eta} \geq -C\frac{|Du|}{\eta}.
\end{equation}
Applying the estimates (\ref{2_6_2}), (\ref{secondterm}), (\ref{thirdterm}), (\ref{fourthterm}) and (\ref{fifthterm}) to (\ref{maxprin_testfunction}), we have at $x^*$ 
\begin{equation*}
    0 \geq -C|Du|-\frac{C}{\eta}|Du|+\frac{C}{M}|Du|^2.
\end{equation*}
The rest of the proof is the same as Theorem \ref{grad_yuan}.
\end{proof}

We restate the Jacobi inequalities established by Zhou for our settings.
\begin{lemma}
\label{jacobi1}
    (Jacobi inequality for $n\geq 3).$ (\cite{Zhou2021} Lemma 4.3). 
    Consider the same setting as Theorem \ref{thm1.1}. Let $n\geq 3$ and let $u$ be a smooth solution to the special Lagrangian equation (\ref{SLE}) with $\theta\geq 0$ and  $\lambda_1 = \lambda_2 = \dots = \lambda_m = \lambda > \lambda_{m+1} $ at a point $p$. Suppose that $\lambda(D^2 u) \subset \overline{\Gamma_{n-1}^{(n-2)/2}}$, then there exist constants $\alpha,\delta>0 $ depending only on $n$ such that if $\lambda>\delta$, then $b_m:= \frac{1}{m}\sum_{i=1}^m \lambda_i$ is smooth near $p$ and 
    \begin{equation}
        \Delta_g b_m(p) \geq (1+ \alpha) \frac{|\nabla_g b_m|^2}{b_m}(p).
    \end{equation}
\end{lemma}
    \begin{lemma}
    \label{jacobi2}
        (Jacobi inequality for $n = 3$). (\cite{Zhou2021} Lemma 3.3). Consider the same setting as Theorem \ref{thm1.2}. Let $n =  3$ and let $u$ be a smooth solution to the special Lagrangian equation (\ref{SLE}) with $\theta\geq 0$ and $\lambda_1 = \lambda_2 = \dots = \lambda_m = \lambda > \lambda_{m+1} $ at a point $p$. Suppose that $\sigma_2(D^2u) \geq (3/5-\varepsilon) \lambda_2 \lambda_3$ for some constant $\varepsilon >0$, then there exist constants $\alpha,\delta>0 $ depending only on $\varepsilon$ such that if $\lambda>\delta$, then $b_m:= \frac{1}{m}\sum_{i=1}^m \lambda_i$ is smooth near $p$ and 
    \begin{equation}
        \Delta_g b_m(p) \geq (1+ \alpha) \frac{|\nabla_g b_m|^2}{b_m}(p).
    \end{equation}
    In particular, we may choose $\alpha = \min(1,\varepsilon)/16$ and $\delta = 4\min(1,\varepsilon)^{-1/2}+\max(1,\varepsilon)$.
    \end{lemma}

We restate Savin's small perturbation theorem (\cite{Savin2007} Theorem 1.3) for equations $F(M)$ depending only on the Hessian, defined on the space of $n \times n$ symmetric matrices.
\begin{theorem}
\label{savin}
    (Savin small perturbation theorem). (\cite{Savin2007} Theorem 1.3). Suppose that $F(D^2 u)$ satisfies the following properties:
    \begin{enumerate} [(a)]
        \item $F$ is elliptic,
        \item $F$ is uniformly elliptic on a $\delta$-neighborhood of $0$,
        \item $F(0)=0$,
        \item $F$ is $C^2$ and $\|D^2F\| \leq K$ in a $\delta$-neighborhood of $0$.
    \end{enumerate}
    There exists a constant $c_1$ such that, if $u$ is a viscosity solution of $F(D^2u)=0$ for $x \in B_1(0)$ with 
    \[\|u\|_{L^\infty(B_1(0))} \leq c_1,\]
    then $u \in C^{2,\alpha}(B_{1/2}(0))$, and
    \[\|u\|_{C^{2,\alpha}(B_{1/2}(0))}\leq \delta.\]
\end{theorem}

\section{Alexandrov-type theorem}
\label{section partial regularity}
In this section, we establish the Alexandrov-type theorems for our settings. Under the assumptions of Theorem \ref{thm1.1} and Theorem \ref{thm1.2}, we cannot directly apply the classical Alexandrov theorem for convex functions, nor the Alexandrov-type theorem for $k$-convex functions, where $k>n/2$ (\cite{Chaudhuri_Trudinger_2005} Theorem 1.1). As a preliminary step, we first observe that $u$ is Lipschitz continuous. The gradient estimates of Lemma \ref{grad_est} further yield the weighted norm Lipschitz inequality. Moreover, by adapting the argument of Theorem 2.4 in \cite{Chaudhuri_Trudinger_2005}, the constraints in Theorem \ref{thm1.1} and Theorem \ref{thm1.2} are sufficient to guarantee the existence of a vector-valued Radon measure $[D^2u] = [\mu^{ij}]$, provided that $u$ is the uniform limit of smooth solutions $u_k$ satisfying the assumptions of Theorem \ref{thm1.1} and Theorem \ref{thm1.2}. Using these three properties, the argument in Section 4 of \cite{10.4007/annals.2025.201.2.4} then yields the desired result. We remark that Shankar-Yuan \cite{10.4007/annals.2025.201.2.4} approach is adapted from Evans–Gariepy \cite{EG} and Chaudhuri–Trudinger \cite{Chaudhuri_Trudinger_2005}. For the reader's convenience, we explicitly note that throughout this section, $u_k$ will always denote the sequence of smooth functions introduced in Proposition \ref{partial_regularity} below. 

\begin{proposition}
\label{partial_regularity}
Let $n\geq 3$ and let $u_k \in C^\infty(B_2(0))$ be a sequence of smooth solutions to the special Lagrangian equation (\ref{SLE}) in $B_2(0)$ with the phase constant $\theta \geq 0$ and
\[\|u_k\|_{C^{0,1}(B_1(0))} \leq A.\] 
Suppose that a subsequence of $u_k$ converges uniformly in $B_1(0)$ to a continuous function $u \in C^0(B_1(0))$, and either 
\begin{enumerate}
[(i)]
    \item  $\lambda(D^2 u_k) \subset  \overline{\Gamma^{(n-2)/2}_{n-1}}$ for all $k \in \mathbb{N}$, or
    \item   $n =3$ and there exists a constant $\varepsilon > 0$ such that 
    \[\sigma_2(D^2 u_k) \geq (3/5-\varepsilon) \lambda_2(D^2u_k)\lambda_3(D^2u_k) \text{ for all } k \in \mathbb{N},\]
\end{enumerate}
then $u$ is twice differentiable almost everywhere in $B_1(0)$, and for almost every $x \in B_1(0)$, there is a quadratic polynomial $Q$ such that
\begin{equation}
    \sup_{y\in B_r(x)} |u(y)-Q(y)| = o(r^2).
\end{equation}
\end{proposition}
\begin{proof}
    Without loss of generality, we may assume that $u_k$ converges to $u$ uniformly. We can immediately see that $u$ is Lipschitz:
    \begin{equation}
    \label{lip}
        \|u\|_{C^{0,1}(B_1(0))} \leq A.
    \end{equation}
    Moreover, by Rademacher's theorem, $u$ is almost everywhere differentiable. 

    By Lemma \ref{grad_est}, we have 
    \begin{equation*}
        |Du_k(0)| \leq C(n,\varepsilon) \frac{osc_{B_{R}(0)}u_k}{R}.
    \end{equation*}
    Following the argument in Corollary 3.4 of \cite{1c5190b3-3a64-33f8-8058-628d6f1053fe}, and using the interpolation argument in Lemma 2.6 of \cite{1c5190b3-3a64-33f8-8058-628d6f1053fe}, we obtain a weighted norm Lipschitz estimate for $u_k$ on $B_r(z) \subset B_1(0)$:
    \begin{equation*}
        \sup_{x,y\in B_r(z), x \neq y} d_{x,y}^{n+1}\frac{|u_k(x)-u_k(y)|}{|x-y|} \leq C \int_{B_r(z)}|u_k| dx,
    \end{equation*}
    where $d_{x,y} = \min(d_x,d_y)$ and $d_x = \dist(x,\partial B_r)$.
    Since $u_k \rightarrow u$ uniformly, we have 
    \begin{equation}
    \label{weighted_lip}
        \sup_{x,y\in B_r(z), x \neq y} d_{x,y}^{n+1}\frac{|u(x)-u(y)|}{|x-y|} \leq C(n) \int_{B_r(z)}|u| dx.
    \end{equation}

    We now show that the weak Hessian $\partial^2 u$, interpreted as a vector-valued distribution, defines a vector-valued Radon measure $[D^2u] = [\mu^{ij}]$
    \begin{equation*}
        \int_{B_1(0)} u(x) \phi_{ij}(x)dx = \int_{B_1(0)} \phi(x) d\mu^{ij}, 
    \end{equation*}
    $\forall \phi \in C^2_c(B_1(0))$. We can divide the proof into three cases: \\
    Case (1). $n \geq 4$.\\
    Case (2). $n=3$ under assumption (i) of Proposition \ref{partial_regularity}.\\
    Case (3). $n=3$ under assumption (ii) of Proposition \ref{partial_regularity}. \\
    We further split Case (3) into two cases. \\
    Case (3a). Suppose that for all $k\in \mathbb{N}$, there is no point $p\in B_1(0)$ such that $\lambda_2(D^2u_k(p))<0 $. \\ 
    Case (3b). There exist some $k \in \mathbb{N}$ and some points $p \in B_1(0)$ such that $\lambda_2(D^2u_k(p)) < 0$. \\
The existence of the vector-valued Radon measure in each case is proved in Lemma \ref{case1}, Lemma \ref{case2}, Lemma \ref{case3a} and Lemma \ref{case3b} respectively.

We have now shown the Lipschitz inequality (\ref{lip}), the weighted norm Lipschitz inequality (\ref{weighted_lip}) and the existence of a vector-valued Radon measure $[D^2u] = [\mu^{ij}]$. Combining these three results, together with the argument in \cite{10.4007/annals.2025.201.2.4} Section 4 yields Proposition \ref{partial_regularity}. We briefly outline the steps. 
The Lipschitz inequality (\ref{lip}) for $u$ implies (4.2) of \cite{10.4007/annals.2025.201.2.4}. The existence of the vector-valued Radon measure $[D^2u] = [\mu^{ij}]$ implies (4.3) and (4.4) of \cite{10.4007/annals.2025.201.2.4}. (4.2), (4.3) and (4.4) in \cite{10.4007/annals.2025.201.2.4} together imply (4.6) in \cite{10.4007/annals.2025.201.2.4} by Lemma 4.1 in \cite{10.4007/annals.2025.201.2.4}. The weighted norm Lipschitz inequality (\ref{weighted_lip}) yields (4.7) in \cite{10.4007/annals.2025.201.2.4}. Putting (4.6) and (4.7) together, we can derive the conclusion of Proposition \ref{partial_regularity} by Lemma 4.2 in \cite{10.4007/annals.2025.201.2.4}.
\end{proof}
We first consider the case $n \geq 4$, where the 2-convexity of the sequence $u_k$ ensures the existence of a vector-valued Radon measure.
\begin{lemma}
\label{case1}
    Under assumption (i) of Proposition \ref{partial_regularity}, suppose that $n\geq 4$, then there exists a vector-valued Radon  measure $[\mu^{ij}]$ such that
    \begin{equation*}
        \int_{B_1(0)} u(x) \phi_{ij}(x)dx = \int_{B_1(0)} \phi(x) d\mu^{ij}, 
    \end{equation*}
    $\forall \phi \in C^2_c(B_1(0))$. 
\end{lemma}
\begin{proof}
    We will show that $\sigma_1(D^2u_k(x)) \geq 0 $ and $\sigma_2(D^2u_k(x)) \geq 0 $, $\forall x \in B_1(0)$, $ \forall k \in \mathbb{N}$. Then $u_k$ is 2-convex, and its uniform limit $u$ is 2-convex as well. By Theorem 2.4 in \cite{Chaudhuri_Trudinger_2005}, there exists a vector-valued Radon  measure $\mu^{ij}= \mu^{ji}$ such that      \begin{equation*}
        \int_{B_1(0)} u(x) \phi_{ij}(x)dx = \int_{B_1(0)} \phi(x) d\mu^{ij}, 
    \end{equation*}
    $\forall \phi \in C^2_c(B_1(0))$.

    For the discussion below, whenever we write $\lambda_i$, it refers to $\lambda_i(D^2u_k(x))$.  By Lemma \ref{n>=3eigenvalues}, $\lambda_i \geq |\lambda_n|$ for any $1\leq i \leq n-1$. We may assume that $\lambda_n <0$. Otherwise the claim is trivial. Hence,
    \[\sigma_1(\lambda_1,\dots,\lambda_n)
    \geq (n-2)|\lambda_n| >0.\]
    Moreover,
    \begin{equation*}
        \sigma_2(\lambda_1,\dots,\lambda_n) = (\lambda_1+\dots+\lambda_{n-1})\lambda_n + \sum_{i< j <n} \lambda_i \lambda_j.
    \end{equation*}
    Note that each term in the first sum is negative, while each term in the second sum is non-negative. 
    By Lemma \ref{n>=3eigenvalues}, the sum of $\lambda_1 (\lambda_2+\dots+ \lambda_{n-1})$ and $(\lambda_1+\lambda_3+\dots+\lambda_{n-1}) \lambda_n$ is non-negative. Also, the sum of  $\lambda_2 \lambda_3$ and $\lambda_2 \lambda_n$ is non-negative as well.
    Then we have the inequality $\sigma_2(\lambda_1,\dots,\lambda_n) \geq 0$. We observe that such an argument does not apply to the case $n = 3$ because $\lambda_3$ could be negative in such a case.
    \end{proof}

Next, we handle the case $n=3$ under assumption (i) of Proposition \ref{partial_regularity}.
\begin{lemma}
\label{case2}
    Under assumption (i) of Proposition \ref{partial_regularity}, suppose that $n=3$, then there exists a vector-valued Radon  measure $[\mu^{ij}]$ such that
    \begin{equation*}
        \int_{B_1(0)} u(x) \phi_{ij}(x)dx = \int_{B_1(0)} \phi(x) d\mu^{ij}, 
    \end{equation*}
    $\forall \phi \in C^2_c(B_1(0))$. 
\end{lemma}
    \begin{proof}        
    Define the dual cone
    \begin{equation}
        (\Gamma_2^{1/2})^*:= \{\lambda \in \mathbb{R}^3 | \text{ }\langle\lambda,\mu \rangle \geq 0 \text{ }\forall \mu \in \Gamma_2^{1/2}\}.
    \end{equation}
    Now define $T_A: C^2_c(B_1(0)) \rightarrow \mathbb{R}$ by
    \begin{equation*}
        T_A(\phi) := \int_{B_1(0)} u(x) \sum_{i,j} a^{ij} D_{ij}\phi(x) dx.
    \end{equation*}
    Here $A = a^{ij}$ is a matrix in $(\Gamma_2^{1/2})^*$. We will show that $T_A(\phi) \geq 0 $ whenever $\lambda(A) \in (\Gamma_2^{1/2})^*$ and $\phi \geq 0$. Observe that the dual cone we defined is symmetric, so the order of eigenvalues does not matter. 
    We observe that
    \begin{align*}
        &\int_{B_1(0)} u_k(x) \sum_{i,j} a^{ij}D_{ij}\phi(x) dx \\
        =&\int_{B_1(0)} \sum_{i,j} D_{ij}u_k(x)  a^{ij}\phi(x) dx \\
        &=\int_{B_1(0)} tr(A D^2u_k(x)) \phi(x) dx  \geq 0.
    \end{align*}
    Here we integrate by parts twice to obtain first equality and use the definition of $(\Gamma_2^{1/2})^*$ together with the assumption that $\phi \geq 0$ to derive the last inequality. Since $u_k$ converges to $u$ uniformly, we have 
    $T_A(\phi) \geq 0 $ for $\lambda(A) \in (\Gamma_2^{1/2})^*$ and $\phi \geq 0$. The remainder of the argument showing the existence of $\mu^{ij}$ as a Radon  measure follows exactly as in Theorem 2.4 in \cite{Chaudhuri_Trudinger_2005}, once we verified that $(1,1,0),(1,0,1),(0,1,1),(1,1,1),(1+t,1-t,1) \in (\Gamma_2^{1/2})^*$ if $0\leq t \leq 1/2$. We will rewrite \cite{Chaudhuri_Trudinger_2005} argument below.

    By the Riesz representation theorem, there exists a Radon measure $\mu^A$ defined on $B_1(0)$ such that 
    \begin{equation}
        T_A(\phi) = \int_{B_1(0)} u(x) \sum_{i,j} a^{ij} \phi_{ij}(x) = \int_{B_1(0)} \phi(x) d\mu^A,
    \end{equation}
    for all $\phi \in C^2_c(B_1(0))$. Define $A_i$ to be the diagonal matrix with all entries 1 but the ith diagonal entry to be 0. 
    We claim that $A_i \in (\Gamma_2^{1/2})^*$.
    Let $(\lambda_1,\lambda_2,\lambda_3) \in \Gamma_2^{1/2}$. By Lemma \ref{n>=3eigenvalues}, we have $\lambda_1 +\lambda_2 \geq 0 $, $\lambda_2 +\lambda_3 \geq 0$ and $\lambda_1 + \lambda_3 \geq 0$. 
    By a similar argument, we can see that $I_3\in  (\Gamma_2^{1/2})^*$, since $\sigma_1(\lambda) \geq 0 $ for $\lambda \in \Gamma_2^{1/2}$. Hence, there exist Radon  measures $I_3$ and $\mu^i$ defined on $B_1(0)$ such that 
    \begin{equation}
        \int_{B_1(0)} u(x) \sum_{j} D_{jj}\phi(x) dx = \int_{B_1(0)}\phi(x) d\mu^{I_3},
    \end{equation}
    and 
    \begin{equation}
        \int_{B_1(0)} u(x) \sum_{j\neq i} D_{jj}\phi(x) dx = \int_{B_1(0)}\phi(x) d\mu^i.
    \end{equation}
    We can define Radon  measures $\mu^{ii} := \mu^{I_3} - \mu^i$, then 
    \begin{equation}
        \int_{B_1(0)} u(x)  D_{ii}\phi(x) dx = \int_{B_1(0)}\phi(x) d\mu^{ii},
    \end{equation}
    for all $\phi \in C^2_c(B_1(0))$. 
    
    We also claim that $(1-t,1+t,1) \in (\Gamma_2^{1/2})^*$ for sufficiently small $t$, say $0\leq t \leq 1/2$. For any $\lambda  \in \Gamma_2^{1/2}$, assuming that $\lambda_1 \geq \lambda_2 \geq \lambda_3 $ where ($\lambda_1,\lambda_2 , \lambda_3$) is an arrangement of $\lambda$, we have
    $\langle \lambda, (1-t,1+t,1) \rangle \geq (1-t)\lambda_1+\lambda_2 + (1+t) \lambda_3 \geq \lambda_1/2 +\lambda_2 + 
    3\lambda_3/2 \geq 0$.
    Now fix $0 < t \leq 1/2$ and $i\neq j \in \{1,\dots,n\}$. Define $A_{ij}:= I_3 + t[e_i \otimes e_j+ e_j \otimes e_i]. $ A direct computation shows the eigenvalues of $A_{ij} $ is $(1+t,1-t,1)$.
    By definition,
    \begin{equation*}
        \sum_{k,l=1}^3a^{kl}D_{kl}\phi(x) =  \sum_{k=1}^3 D_{kk}\phi(x)+2tD_{ij} \phi(x).
    \end{equation*}
    Define Radon measures $\mu^{A_{ij}}$ by
    \begin{equation*}
        \int_{B_1(0)} \phi(x)d\mu^{A_{ij}} = \int_{B_1(0)} u(x)(D_{kk}\phi(x)+2tD_{ij} \phi(x)).
    \end{equation*}
    We have
    \begin{align*}
        \int_{B_1(0)} u(x) D_{ij} \phi(x) &= \frac{1}{2t} (\int_{B_1(0)} u(x) \sum_{k,l=1}^3 a^{kl}D_{kl}\phi(x) dx \\
        &- \int_{B_1(0)} u(x) \sum_{k=1}^3 D_{kk}\phi(x) dx) \\
        &=\frac{1}{2t}(\int_{B_1(0)}\phi(x) d\mu^{A_{ij}}- \int_{B_1(0)} \phi(x) d \mu^{I_3}) \\
        &= \int_{B_1(0)} \phi(x) d \mu^{ij},
    \end{align*}
    where we define the measure $\mu^{ij} = \frac{1}{2t}(\mu^{A_{ij}}-\sum_{k=1}^n \mu^{kk})$.
    \end{proof}
We now consider the first subcase of assumption (ii) with $n=3$, where no eigenvalue $\lambda_2(D^2 u_k)$ is negative. The argument is similar to Lemma~\ref{case2}.
\begin{lemma}
\label{case3a}
    Under assumption (ii) of Proposition \ref{partial_regularity}, suppose further that   for all $k\in \mathbb{N}$, there is no point $p\in B_1(0)$ such that $\lambda_2(D^2u_k(p))<0 $, then there exists a vector-valued Radon  measure $[\mu^{ij}]$ such that
    \begin{equation*}
        \int_{B_1(0)} u(x) \phi_{ij}(x)dx = \int_{B_1(0)} \phi(x) d\mu^{ij}, 
    \end{equation*}
    $\forall \phi \in C^2_c(B_1(0))$. 
\end{lemma}
\begin{proof}
    By Lemma \ref{n=3eigenvalues}(a), we have $\lambda_1+\lambda_2 \geq 0$, $\lambda_1+\lambda_3 \geq 0$ and $\lambda_2+\lambda_3 \geq 0$. Note that the inner product between any of the vectors $(1,1,0),(1,0,1),\\(0,1,1),(1,1,1),(1+t,1-t,1)$ and the vector of eigenvalues  $(\lambda_1,\lambda_2,\lambda_3)$ of $D^2u_k$ at any points $p \in B_1(0)$ will be non-negative if $0\leq t \leq 1/2$. We may use the argument of Lemma \ref{case2}.
    \end{proof}
Finally, we treat the remaining subcase of assumption (ii) with $n=3$, where some $\lambda_2(D^2 u_k)$ may be negative. A suitable modification of $u_k$ ensures the construction of the measure.    
\begin{lemma}
\label{case3b}
    Under assumption (ii) of Proposition \ref{partial_regularity}, suppose further that there exist some $k \in \mathbb{N}$ and some points $p \in B_1(0)$ such that $\lambda_2(D^2u_k(p)) < 0$, then there exists a vector-valued Radon  measure $[\mu^{ij}]$ such that
    \begin{equation*}
        \int_{B_1(0)} u(x) \phi_{ij}(x)dx = \int_{B_1(0)} \phi(x) d\mu^{ij}, 
    \end{equation*}
    $\forall \phi \in C^2_c(B_1(0))$. 
\end{lemma}
\begin{proof} 
First, we observe that $3/5- \varepsilon < 0 $. If not, for any $k\in\mathbb{N}$ and $p\in B_1(0)$ such that $\lambda_2(D^2u_k(p))  < 0 $, we have
$\sigma_2(D^2u_k) \geq (3/5-\varepsilon)\lambda_2(D^2u_k) \lambda_3(D^2u_k) \geq  0$. This is impossible, since $\sigma_2(D^2u_k) \geq 0 $ and $\sigma_1(D^2u_k) \geq 0$ (by Lemma \ref{n=3eigenvalues}(d)) imply $\lambda_2(D^2u_k) + \lambda_3(D^2u_k) \geq 0 $. However, this would force \\ $\lambda_2(D^2u_k(p)) \geq 0$.
    
    By Lemma \ref{n=3eigenvalues}(b) and (c), for any $k\in\mathbb{N}$ and $p\in B_1(0)$ such that \\ $\lambda_2(D^2u_k(p))  < 0 $, we have $|\lambda_i(D^2u_k(p))| \leq \lambda_{\max}(D^2u_k(p))\leq \max(1,\varepsilon)$ for any $1\leq i\leq 3$. For all $k \in \mathbb{N}$, we define  $\Tilde{u}_k = u_k + \frac{1}{2}\max(1,\varepsilon) \sum_{i=1}^n x_i^2$ and $\Tilde{u} = u + \frac{1}{2}\max(1,\varepsilon) \sum_{i=1}^n x_i^2$.
    Note that $\Tilde{u}_k$ may not satisfy the special Lagrangian equation (\ref{SLE}), but this is not required. 

    For any $k\in\mathbb{N}$ and $p\in B_1(0)$ such that $\lambda_2(D^2u_k(p))  < 0 $, we have $\lambda_1(D^2 \Tilde{u}_k(p))\geq \lambda_2(D^2 \Tilde{u}_k(p))\geq \lambda_3(D^2 \Tilde{u}_k(p)) \geq 0.$ For these $k,p$, the inner product between any of the vectors $(1,1,0),(1,0,1),(0,1,1), (1,1,1),\\(1+t,1-t,1)$ and the vector of eigenvalues $(\lambda_1,\lambda_2,\lambda_3)$ of $D^2\Tilde{u}_k(p)$ will be non-negative if $0\leq t \leq 1/2$.

    For any $k\in \mathbb{N}$ and $p\in B_1(0)$ such that $\lambda_2 (D^2 u_k(p)) \geq 0$, by Lemma \ref{n=3eigenvalues}(a), we have $\lambda_i(D^2\Tilde{u}_k(p))+\lambda_j(D^2\Tilde{u}_k(p)) \geq \lambda_i(D^2u_k(p))+\lambda_j(D^2u_k(p)) \geq 0$ for any $i\neq j \in \{1,2,3\}$. The inner product between any of the vectors $(1,1,0),(1,0,1),(0,1,1),(1,1,1),(1+t,1-t,1)$ and the vector of eigenvalues \\$(\lambda_1,\lambda_2,\lambda_3)$ of $D^2\Tilde{u}_k(p)$ will be non-negative if $0\leq t \leq 1/2$. 
    
  Using the argument in Lemma \ref{case2}, there exists a vector-valued Radon measure $[D^2\Tilde{u}] = [\Tilde{\mu}^{ij}]$:
    \begin{equation*}
        \int_{B_1(0)} \Tilde{u}(x) \phi_{ij}(x)dx = \int_{B_1(0)} \phi(x) d\Tilde{\mu}^{ij}. 
    \end{equation*}

Note that $\frac{1}{2}\max(1,\varepsilon) \sum_{i=1}^3 x_i^2$ is a convex function. Therefore, we can find a vector-valued Radon measure for it. Subtracting such a vector-valued measure from $[\Tilde{\mu}^{ij}]$ gives a vector-valued Radon measure $[D^2u] = [\mu^{ij}]$:
    \begin{equation*}
        \int_{B_1(0)} u(x) \phi_{ij}(x)dx = \int_{B_1(0)} \phi(x) d\mu^{ij} .
    \end{equation*}

    \end{proof}

\section{Doubling inequality}
\label{section doubling inequality}
In this section, we establish the doubling inequalities for our settings. The argument in \cite{Shankar2024HessianEF} Section 4 applies here with minor modifications. First, if there exists a point $p$ such that $\lambda_{\max}(D^2u(p))$ is not sufficiently large, the Jacobi inequalities in Lemma \ref{jacobi1} and Lemma \ref{jacobi2} may not hold at $p$. However, $\lambda_{\max}(D^2u(p))$ is bounded in this case. Another difference is that, in the setting of Theorem \ref{thm1.2}, there could be two negative eigenvalues at a point $p$. However, by Lemma \ref{n=3eigenvalues}(b) and (c), $|D^2u(p)|$ is bounded in this case.
\begin{proposition}
\label{doubling1}
Let $n\geq 3$ and let $u$ be a smooth solution to the special Lagrangian equation (\ref{SLE}) in $B_2(0)$ with the phase constant $\theta \geq 0$. Suppose that $\lambda(D^2 u) \subset  \overline{\Gamma^{(n-2)/2}_{n-1}}$, then for any $y \in B_{1/2}(0)$ and $r<1/4$, we have
\begin{equation*}
    \sup_{B_{1/4}(y)} \lambda_{\max}(D^2u) \leq C \sup_{B_r(y)} \lambda_{\max}(D^2u)+C,
\end{equation*}
where $C$ is a constant depending on $r,n$ and $\|u\|_{C^{0,1}(B_1(0))}$.
\end{proposition}
\begin{proof}
    Let $\alpha, h^{-1} \gg 1$ be two constants to be chosen later. We choose the same cutoff function $\eta(x)$ as in \cite{Shankar2024HessianEF} Section 5:
    \begin{equation}
        \eta(x) = ( e^{(S- \varphi(x))/h} - 1)_+,
    \end{equation} where 
    \begin{equation}
        \varphi = (x-y)\cdot Du - u +u(y) - \frac{\alpha^{-1}2^\alpha}{|x-y|^{2\alpha}},
    \end{equation}
    and 
    \begin{equation}
        S = -1 - \|(x-y)\cdot Du - u+u(y)\|_{L^\infty(B_{1/2}(y))} - \alpha^{-1}2^{3\alpha}.
    \end{equation}
    The choice of cutoff function is motivated by Korevaar exponential type cutoff function \cite{Korevaar1987} and Guan-Qiu \cite{GuanQiu2019} type radial derivative $(x-y)\cdot Du - u$. Observe that $S-\varphi<0$ on $\partial B_{1/2}(y)$ and $S-\varphi>0$ on $B_{1/4}(y)$ if $\alpha$ is large enough. 
    
    Let $p$ be the maximum point of $\eta(x) b_1(x)$ on $B_{1/2}(y) \backslash B_{r}(y) $ where $b_1(x)$ is defined in Lemma \ref{jacobi1}. Note that the function $b_1(x)$ is different with the function $a_1(x)$ in \cite{Shankar2024HessianEF} Section 5 because we have a different function for the Jacobi inequality. 
    
     If $p$ lies in the interior of $B_{1/2}(y) \backslash B_{r}(y)$, then there are two cases: \\
     Case (a). $\lambda_{\max}(p) \leq \delta$. \\
     Case (b). $\lambda_{\max}(p) > \delta$, 
     where $\delta$ is the constant in Lemma \ref{jacobi1}. 
     
    Case (a). The Jacobi inequality in Lemma \ref{jacobi1} does not apply in this case. However, by Lemma \ref{n>=3eigenvalues}, $|D^2u(p)|$ is already bounded. Note that we have an upper bound and a lower bound for $\eta$ depending on $r,n$ and $\|u\|_{C^{0,1}(B_1(0))}$. We obtain the desired bound 
    \begin{equation}
    \label{doublinga}
        \begin{split}
        \sup_{B_{1/4}(y)}b_1(x) & \leq \sup_{B_{1/4}(y)\backslash B_r(y)}b_1(x) + \sup_{B_{r}(y)}b_1(x) \\
        &\leq C  \sup_{B_{1/4}(y)\backslash B_r(y)}\eta(x)b_1(x) + \sup_{B_{r}(y)}b_1(x) \\
        &\leq C \sup_{B_{1/4}(y)\backslash B_r(y)}\eta(x) \delta +  \sup_{B_r(y)} b_1(x) \\
        &\leq C+ \sup_{B_r(y)} b_1(x), 
        \end{split}
    \end{equation}
    where the constant $C$ depends on $r,n$ and $\|u\|_{C^{0,1}(B_1(0))}$. Notice that the constant $C$ depends on the constants $\alpha$ and $h$ but they will be fixed in Case (b).
    
Case (b).  By Lemma \ref{jacobi1}, the Jacobi inequality holds. By Lemma \ref{n>=3eigenvalues}, we have $\lambda_i \geq  |\lambda_n|$ for $1 \leq i \leq n-1$, and $\lambda_n$ is the only eigenvalue that can be negative. We observe that the entire argument in \cite{Shankar2024HessianEF} Section 5 Steps 2 and 3 still applies because these are the only properties of the condition $\theta \geq (n-2)\pi/2$ that were used. We derive that Case (b) cannot occur if $\alpha$ and $h^{-1}$ are large enough. 

If $p$ lies on the boundary of $B_{1/2}(y) \backslash B_{r}(y) $. Since $\eta(x) = 0 $ on $\partial B_{1/2}(y)$, we have 
    \begin{equation}
    \label{doublingb}
        \sup_{B_{1/2}(y)\backslash {B_r}(y)} \eta(x) b_1(x) = \sup_{\partial{B_r}(y)} \eta(x) b_1(x) \leq C(r,n,\|u\|_{C^{0,1}(B_1(0))}) \sup_{B_r(y)}b_1(x).
    \end{equation}  
    We then derive 
    \begin{equation}
    \label{doublingc}
        \begin{split}
        \sup_{B_{1/4}(y)}b_1(x) & \leq \sup_{B_{1/4}(y)\backslash B_r(y)}b_1(x) + \sup_{B_{r}(y)}b_1(x) \\
        &\leq C  \sup_{B_{1/4}(y)\backslash B_r(y)}\eta(x)b_1(x) + \sup_{B_{r}(y)}b_1(x) \\
        &\leq C \sup_{B_r(y)} b_1(x) +  \sup_{B_r(y)} b_1(x) \\
        &\leq C \sup_{B_r(y)} b_1(x). 
        \end{split}
    \end{equation}
    Hence, either (\ref{doublinga}) or (\ref{doublingc}) holds. We can obtain the doubling inequality stated in Proposition \ref{doubling1}.

\end{proof}

\begin{proposition}
\label{doubling2}
   Let $n=3$ and let $u$ be a smooth solution to the special Lagrangian equation (\ref{SLE}) in $B_2(0)$ with the phase constant $\theta \geq 0$. Suppose that $\sigma_2(D^2u) \geq (3/5-\varepsilon) \lambda_2 \lambda_3$ for some constant $\varepsilon >0$ where $\lambda_1 \geq \lambda_2 \geq \lambda_3 $ are eigenvalues of $D^2 u$,  then for any $y \in B_{1/2}(0)$ and $r<1/4$, 
\begin{equation*}
    \sup_{B_{1/4}(y)} \lambda_{\max}(D^2u) \leq C \sup_{B_r(y)} \lambda_{\max}(D^2u)+C,
\end{equation*}
where $C$ is a constant depending on $r, \|u\|_{C^{0,1}(B_1(0))}$ and $\varepsilon$.
\end{proposition}
\begin{proof}
The proof is nearly identical to that of Proposition \ref{doubling1} except we have an extra case $\lambda_2(p) < 0$. By Lemma \ref{n=3eigenvalues}(b) and (c), we can obtain the bound $|D^2u(p)| \leq \max(1,\varepsilon)$. Hence, the same argument as in Case (a) in Proposition \ref{doubling1} works after replacing the constant $\delta$ by $\max(1,\varepsilon)$.
\end{proof}
\section{Proofs of Theorems \ref{thm1.1} and \ref{thm1.2}}
\label{section main theorem proof}

\begin{proof}[Proof of Theorem \ref{thm1.1}] 
The proof of Theorem \ref{thm1.1} follows the same argument as in \cite{Shankar2024HessianEF} Section 6, except that we have an additional constant term in the doubling inequality. To clarify the doubling framework, we briefly sketch the compactness argument.

Let $u_k$ be a sequence of functions that satisfy the assumptions of Theorem \ref{thm1.1} with $\|u_k\|_{C^{1,1}(B_2(0))} \leq A$ but $|D^2u_k(0)| \rightarrow \infty$. Choose a subsequence of $u_k$ that converges uniformly in $B_1(0)$ to a continuous function $u \in C^0(B_1(0))$. \\
Step 1: Partial regularity of the limit. By the Alexandrov-type theorem (Proposition \ref{partial_regularity}), we choose a point $y \in B_{1/4}(0)$ such that $u$ is twice differentiable at $y$ and let $Q$ be the quadratic part of $u$ at $y$ such that
\begin{equation*}
    \sup_{x\in B_r(y)} |u(x)-Q(x)| \leq r^2 \sigma(r),
\end{equation*}
where $\sigma(r) = \frac{o(r^2)}{r^2}$ as $r \rightarrow 0$.
We note that $Q$ solves the special Lagrangian equation (\ref{SLE}). \\
Step 2: Flattening the error.
As in \cite{Shankar2024HessianEF}, consider the rescaled function 
\begin{equation*}
    \bar{v}_k(\bar{x}):= \dfrac{1}{r^2}\left(u_k(y+r \bar{x}) - Q(y+r \bar{x}) \right), \quad \bar{x} \in B_1(0).
\end{equation*}
By the triangle inequality, uniform convergence and the Alexandrov-type theorem (Proposition \ref{partial_regularity}), we have:
\begin{equation}
\label{error}
\begin{split}
    \|\bar{v}_k\|_{L^{\infty}(B_1(0))}&
    \le r^{-2}\|u_k(y+r\bar{x})-u(y+r\bar{x})\|_{L^{\infty}(B_1(0))} \\
    &+ \left\|\frac{u(y+r\bar{x})-Q(y+r\bar{x})} {r^2}\right\|_{L^{\infty}(B_1(0))} \\
    &\le r^{-2}o(k)/k + \sigma(r).
\end{split}
\end{equation}
Step 3: Savin stability of partial regularity.
Observe that $\bar{v}_k$ solves the fully nonlinear elliptic PDE on $B_1(0)$:
\begin{equation*}
G(D^2 \bar{v}_k) = \sum_{i=1}^{n} \left[ \arctan \lambda_i(D^2 Q + D^2 \bar{v}_k) - \arctan \lambda_i(D^2 Q) \right] = 0.
\end{equation*}
The operator $G$ satisfies the assumptions of Savin's small perturbation theorem (Theorem \ref{savin}). By \eqref{error}, we can choose a sufficiently small $r = r(\sigma)$ and a sufficiently large $k$ such that $\|\bar{v}_k\|_{L^{\infty}(B_1(0))} \leq c_1$, where $c_1$ is the constant given by Theorem \ref{savin}. Applying Theorem \ref{savin} to $\bar{v}_k$ and rescaling back, we obtain
\begin{equation}
\label{small_ball_error}
    \|u_k\|_{C^{2,\alpha}(B_r(y))} \le C(n,Q,\sigma).
\end{equation}\\
Step 4: Doubling to propagate the partial regularity.
By Proposition \ref{doubling1}, the choice of $r=r(\sigma)$ and (\ref{small_ball_error}), we have
\begin{equation*}
    \begin{split}
        \sup_{B_{1/4}(y)}|D^2u_k| &\leq C(r,n,A,\sup_{B_r(y)}|D^2u_k|) + C(r,n, A) \\
        &\leq C(r,n,A,C(n,Q,\sigma)) + C(r,n, A)\\
        &\leq C(n,A,Q,\sigma).
    \end{split}
\end{equation*}
Since $y \in B_{1/4}(0)$, we have
\begin{equation*}
    |D^2u_k(0)|\leq C(n,A,Q,\sigma).
\end{equation*}
A contradiction arises.
\end{proof}

\begin{proof} [Proof of Theorem \ref{thm1.2}] The proof of Theorem \ref{thm1.2} is exactly the same with Theorem \ref{thm1.1} except that we use the doubling inequality in Proposition \ref{doubling2} instead of the one in Proposition \ref{doubling1}. And the constant $C$ depends on $\varepsilon$ as well.
\end{proof}

\section{Appendix}
\label{appendix}
\begin{theorem}
\label{grad_yuan}
    (Gradient estimate when $|\theta| \geq (n-2)\pi/2$). (\cite{private_note} Theorem 5.1.1).
    Let $n\geq 3$. Suppose that $u$ is a smooth solution to the special Lagrangian equation (\ref{SLE}) on $B_{R}(0)$ with $|\theta| \geq (n-2)\pi/2$, then 
\begin{equation}
    |Du(0)| \leq C(n)\frac{osc_{B_{R}(0)}u}{R}.
\end{equation}

\end{theorem}
\begin{proof}
    We may assume that $R= 1 $ by rescaling $u(Rx)/R^2$. Note that
\[D^2(\frac{u(Rx)}{R^2}) = D^2u(Rx).\]
Then $\frac{u(Rx)}{R^2}$ satisfies the special Lagrangian equation with the same phase constant $\theta$.
Suppose we have the estimate 
\begin{equation*}
    |D(\frac{u(R  x)}{R^2})(0)| \leq C(n)osc_{B_1(0)}(\frac{u(Rx)}{R^2}),
\end{equation*}
then we can simplify the equation and obtain the desired estimate. We may also assume $\theta \geq (n-2)\pi/2$ by symmetry.

Let \[M:=osc_{B_1(0)} u.\]
We may assume $M>0$. We may replace $u$ by $\Tilde{u}:= u-\min_{B_1(0)}u+M$. We then have $M\leq u \leq 2M$ in $B_1(0)$.
Consider the test function
\begin{equation*}
    w = (1-|x|^2)|Du| + \frac{n}{M}u^2.
\end{equation*}
Write $\eta := 1-|x|^2$.
Suppose that the maximum of $w$ lies on the boundary, say at $x^*$ where $|x^*| = 1$, then
\begin{equation*}
    2nM \geq w(x^*) = \frac{n}{M}u^2(x^*) \geq w(0) = |Du(0)|+\frac{n}{M}u^2(0) \geq |Du(0)|.
\end{equation*}

If the maximum point $x^*$ lies in the interior, we may assume $|Du(x^*)|>0$, since otherwise we can obtain the desired bound. We may further assume there exists a $k \in{\{1,\dots,n\}}$ such that $u_k \geq \frac{|Du|}{\sqrt{n}}$ at $x^*$. If not, we may consider the function \[\Tilde{u}(x_1,\dots,x_k,\dots,x_n) = u(x_1,\dots,-x_k,\dots,x_n),\] which satisfies the special Lagrangian equation with the same phase constant $\theta$. At $x^*$, we have
\begin{equation}
\label{2_3}
    0 = w_i = \eta|Du|_i+\eta_i|Du|+2\frac{n}{M}uu_i.
    \end{equation}
A direct computation gives
\begin{equation*}
    |Du|_i = \frac{\sum_j u_j u_{ij}}{|Du|}.
\end{equation*}
Assume that $D^2u$ is diagonalized at $x^*$, we have
\begin{equation}
\label{2_4}
    \frac{u_i u_{ii}}{|Du|} = |Du|_i = -\frac{2\eta_i|Du|+2uu_i\frac{n}{M}}{\eta}.
\end{equation}
In particular, for the $k\in \{1,\dots,n\}$ chosen just now, 
\begin{align*}
     2\eta_k|Du|+2uu_k\frac{n}{M}&\geq -2|Du|+2u u_k \frac{n}{M}\\
     &\geq -2|Du| + 2M\frac{|Du|}{\sqrt{n}}{\frac{n}{M}} =2(\sqrt{n}-1)|Du| >0.
\end{align*}
Using (\ref{2_4}) and the fact that $u_k >0 $,
we have $\lambda_k = u_{kk} <0 $. Also, by Lemma 2.1 of \cite{Wang2014}, $|\lambda_n| \leq \lambda_i$ for $i \in \{1,\dots,n-1\}$. Hence, we have $\lambda_k = \lambda_n$.

At the maximum point $x^*$, we have
\begin{equation}
\label{2_5}
    0 \geq \Delta_g w = \Delta_g \eta |Du| + 2g^{\alpha\beta} \eta_\alpha |Du|_\beta+ \eta \Delta_g |Du| + 2\frac{n}{M}g^{\alpha \beta}u_\alpha u_\beta+ 2\frac{n}{M}u\Delta_g u.
\end{equation}
We look at each term separately. We will use diagonal coordinates at $x^*$. Consider the first term, we have
\begin{equation*}
    \Delta_g \eta = -2\sum_\alpha g^{\alpha\alpha} \geq -2ng^{nn},
\end{equation*}
where we use the fact that $|\lambda_n|\leq \lambda_i$ for any $1 \leq i \leq n-1$. Hence, 
\begin{equation}
\label{2_6}
    \Delta_g \eta |Du| \geq -2ng^{nn}|Du|.
\end{equation}
Using (\ref{2_4}), the inequality $|\lambda_n|\leq \lambda_i$ for $1\leq i \leq n-1$, and the bound $M\leq u \leq 2M$ for the second term, we have
\begin{equation}
\label{secondterm2}
\begin{split}
\sum_{\alpha,\beta}2g^{\alpha\beta} \eta_\alpha |Du|_\beta &= \sum_\alpha2g^{\alpha\alpha}(-2x_\alpha)(\frac{2x_\alpha|Du|+2\frac{n}{M}uu_\alpha}{\eta}) \\
&\geq -8g^{nn}(\sum_\alpha x_\alpha^2)\frac{|Du|}{\eta}-8n^2 g^{nn}\frac{|Du|}{\eta}\\
&\geq -8g^{nn}\frac{|Du|}{\eta}-8n^2 g^{nn}\frac{|Du|}{\eta}.
\end{split}
\end{equation}
Directly computing the third term, we have
\begin{align*}
    \Delta_g |Du| &= \sum_{\alpha,\beta} g^{\alpha\beta}\partial_\alpha\partial_\beta|Du|\\
    &= \sum_{\alpha,\beta} g^{\alpha \beta} (\frac{\sum_i u_{i\alpha}u_{\beta i}-\sum_i u_{\alpha\beta i}u_i}{|Du|}-\frac{\sum_i u_i u_{\beta i} \sum_j u_j u_{\alpha j}}{|Du|^3}) \\
    &= \sum_{\alpha,\beta} g^{\alpha \beta} (\frac{\sum_i u_{i\alpha}u_{\beta i}}{|Du|}-\frac{\sum_i u_i u_{\beta i} \sum_j u_j u_{\alpha j}}{|Du|^3}),
\end{align*}
where we use the fact that $\sum_{\alpha,\beta}g^{\alpha\beta}u_{\alpha\beta i} = 0$.
Suppose that, at the point $x^*$, $D^2u$ is diagonal, then
\begin{align*}
    \Delta_g |Du| &= \sum_{\alpha} g^{\alpha \alpha} \frac{ u_{\alpha\alpha}^2|Du|^2- u_\alpha^2 u_{\alpha \alpha}^2 }{|Du|^3} \geq 0.
\end{align*}
Hence, 
\begin{equation}
\label{thirdterm}
    \eta \Delta_g |Du| \geq 0.
\end{equation}
Recall that $u_n \geq \frac{|Du|}{\sqrt{n}}$ at $x^*$ and $|\lambda_n| \leq \lambda_i$ for $i \in \{1,\dots,n-1\}$, we can estimate the fourth term by 
\begin{equation}
\label{fourthterm}
    \sum_\alpha\frac{2n}{M}g^{\alpha\alpha}u_\alpha^2 \geq \frac{2n}{M}g^{nn}u_n^2 \geq \frac{2n}{M}g^{nn}\frac{|Du|^2}{n} = \frac{2}{M}g^{nn}|Du|^2.
\end{equation}
For the fifth term, using the fact that $\lambda_i\geq 0$ for $0\leq i\leq n-1$ and (\ref{2_4}), we have
\begin{equation}
\label{fifthterm2}
\begin{split}
    \Delta_g u=  \sum_\alpha g^{\alpha\alpha}u_{\alpha\alpha} \geq g^{nn}u_{nn}= -g^{nn}\frac{|Du|(\eta_n|Du|+2\frac{n}{M}u u_n)}{u_n\eta} \geq -g^{nn}\frac{6n|Du|}{\eta}.
\end{split}
\end{equation}
Applying the estimates (\ref{2_6}), (\ref{secondterm2}), (\ref{thirdterm}), (\ref{fourthterm}) and (\ref{fifthterm2}) to (\ref{2_5}), we have 
\begin{equation*}
    0 \geq g^{nn}(-C|Du|-\frac{C}{\eta}|Du|+\frac{C}{M}|Du|^2).
\end{equation*}
Hence, we have
\begin{equation*}
    \eta(x^*)|Du(x^*)| \leq CM,
\end{equation*}
and
\begin{equation*}
    |Du(0)| \leq w(0) \leq w(x^*) = \eta|Du(x^*)|+\frac{n}{M}u^2(x^*) 
    \leq CM.
\end{equation*}
\end{proof}
\section*{Acknowledgements}
The author would like to thank Yu Yuan for granting permission to use the gradient estimate from his private notes, and Yi Wang for her valuable guidance and supervision. The author also thanks the anonymous referee for helpful comments.
\bibliographystyle{plainnat}
\bibliography{refs}

\end{document}